\documentclass[a4paper,10pt]{amsart}

\usepackage[utf8]{inputenc}
\usepackage{amssymb}
\usepackage{amsmath}
\usepackage{amsthm}
\usepackage{color}
\usepackage{enumerate}

\newcommand{\mat}[4]{ \left ( \begin{array}{cc} #1 & #2 \\ #3 &
      #4 \end{array} \right)} 

\newcommand{\Hom}{\operatorname{Hom}}

\newcommand{\End}{\operatorname{End}}
\DeclareMathOperator{\Par}{Par}
\newcommand{\Aut}{\operatorname{Aut}}

\newcommand{\Ker}{\operatorname{Ker}}

\newcommand{\Img}{\operatorname{Im}}
\newcommand{\Coker}{\operatorname{Coker}}

\newcommand{\Modr}[1]{\mathrm{Mod}\textrm{-}{#1}}

\DeclareMathOperator{\Z}{Z}

\theoremstyle{plain}
\newtheorem{theorem}{Theorem}[section]
\newtheorem{lemma}[theorem]{Lemma}
\newtheorem{proposition}[theorem]{Proposition}
\newtheorem{corollary}[theorem]{Corollary}

\newtheorem{definition}[theorem]{Definition}

\newtheorem{remark}[theorem]{Remark}
\newtheorem{remarks}[theorem]{Remarks}

\newtheorem{example}[theorem]{Example}
\newtheorem{examples}[theorem]{Examples}

\usepackage{tikz-cd}

\title{ENDOMORPHISM RINGS VIA MINIMAL MORPHISMS}
\author[M. Cort\'es]{Manuel Cort\'es-Izurdiaga}
\address{Departamento de Matem\'aticas, Universidad de Almer\'{\i}, Almer\'{\i}a, 04120, Spain}
\email{mizurdia@ual.es}\thanks{The first author is partially supported by the Spanish Government under
   grants MTM2016-77445-P and MTM2017-86987-P which include FEDER funds of the EU and grant UAL18-FQM-B008A-E(UAL/CECEU/FEDER)}
\author[P. A. Guil]{Pedro A. Guil Asensio}
\address{Departamento de Mathematicas, Universidad de Murcia, Murcia, 30100, Spain}
\email{paguil@um.es}\thanks{The
   second author is partially supported by the Spanish Government under
   grant MTM2016-77445-P which includes FEDER funds of the EU, and by Fundaci\'on S\'eneca of
   Murcia under grant
  19880/GERM/15}
\author[D. Kesk\.{i}n]{D. Kesk\.{i}n T\"ut\"unc\"u}
\address{Department of Mathematics, Hacettepe University, Ankara, 06800, Turkey}
\email{keskin@hacettepe.edu.tr}
\author{Ashish K. Srivastava}
\address{Department of Mathematics and Statistics, Saint Louis University, St.
Louis, MO-63103, USA} \email{ashish.srivastava@slu.edu}\thanks{The fourth author is partially supported by a grant from Simons
   Foundation (grant number 426367).}

\begin{document}

\begin{abstract}
We prove that if $u:K \rightarrow M$ is a left minimal extension, then there exists an isomorphism between two subrings, $\End_R^M(K)$ and $\End_R^K(M)$ of $\End_R(K)$ and $\End_R(M)$ respectively, modulo their Jacobson radicals. This isomorphism is used to deduce properties of the endomorphism ring of $K$ from those of the endomorphism ring of $M$ in certain situations such us when $K$ is invariant under endomorphisms of $M,$ or when $K$ is invariant under automorphisms of $M$.
\end{abstract}

\maketitle

\section*{Introduction}

\noindent Let $R$ be a not necessarily commutative ring with unit element and $M$ be a right $R$-module. Suppose that we want to study the endomorphism ring of $M$. One way to do this is to consider an approximation of $M$ by a class of modules $\mathcal X$, from which we know the structure of the endomorphism ring of its objects, and try to deduce properties of $\End_R(M)$ from those of the endomorphism ring of its $\mathcal X$-approximation.

This approach has been fruitfully used to study the endomorphism ring of some classical classes of module. For instance, it is well known that the endomorphism ring of an injective module is von Neumann regular and right self injective modulo the Jacobson radical, and idempotents lift modulo the Jacobson radical. Extending this structural result, Faith and Utumi proved in \cite{FaithUtumi} that the endomorphism ring of a quasi-injective module enjoys these properties as well, using the fact that quasi injective modules are, precisely, those which are invariant under the action of endomorphisms of their injective envelopes \cite{JohnsonWong}. 

Later on, Guil Asensio and Srivastava proved in \cite{GuilSrivastava} that if $M$ is invariant under automorphisms of its injective envelope, then its endomorphism ring is a von Neumann regular ring modulo its Jacobson radical $J$, idempotents lift modulo $J$, and $J$ consists of those endomorphism of $M$ which have essential kernel. This is another example in which properties of the endomorphism ring of the injective envelope of $M$ are transferred to the endomorphism ring of $M$.

This last result has been extended in \cite{GuilKeskinSrivastava} to approximations with respect to a class of modules $\mathcal X$ in the following way: if $M$ has an $\mathcal X$-envelope $M \rightarrow X$ such that $M$ is invariant under automorphisms of $X$, and $\End_R(X)$ is von Neumann regular modulo the Jacobson radical $J$, and idempotents lift modulo $J$, then the endomorphism ring of $M$ shares these properties.

The common situation in all these results is that there exists a morphism $u:M \rightarrow X$, which is minimal in some sense and such that $M$ is invariant under a set of endomorphisms of $M$, and it is possible to deduce the properties of the endomorphism ring of $M$ from those of the endomorphism ring of $X$. 

In this paper we study the relationship between the endomorphisms ring of two modules $K$ and $M$ when there exists a left minimal monomorphism $u:K \rightarrow M$ (see Definition \ref{d:LeftMinimal} for the notion of left minimal morphism). We prove, in Theorem \ref{t:IsomorphismEndomorphismRing}, that there exists an isomorphism between two subrings, $\End_R^M(K)$ and $\End_R^K(M)$ of $\End_R(K)$ and $\End_R(M)$ respectively (see Definition \ref{d:Subrings}), modulo their Jacobson radical. These subrings coincide with the full endomorphisms rings when $u$ is injective with respect to $K$ and $K$ is invariant under endomorphisms of $M$ and, consequently, in this case we obtain an isomorphism between $\End_R(K)$ and $\End_R(M)$ modulo their Jacobson radicals. This result explains the structure of the endomorphism ring of a quasi-injective module obtained by Faith and Utumi or, more generally, of the endomorphism ring of a module which is invariant under endomorphisms of its $\mathcal X$-envelope for some class of modules $\mathcal X$ (see Corollary \ref{c:FinalCorollary}).

When the module $K$ is invariant under automorphisms of $M$ in the left minimal morphism $u:K \rightarrow M$, then there does not exist, in general, an isomorphism between the endomorphisms ring of $K$ and $M$ modulo their Jacobson radical. However, the second part of our main Theorem \ref{t:IsomorphismEndomorphismRing}, which states that the aforementioned subring $\End_R^M(K)$ of $\End_R(K)$ modulo the Jacobson radical is isomorphic to a subring of $\End_R(M)$ modulo the Jacobson radical, allows us to transfer, in Theorem \ref{t:IsomorphismEndomorphismRingAutoInvariant}, properties from $\End_R(M)$ to $\End_R^M(K)$. As above, if the monomorphism $u$ is $K$-injective, then $\End_R^M(K)$ is the whole endomorphism ring of $K$, and, again, we can obtain the structure of $\End_R(K)$ from the structure of $\End_R(M)$.

There are other remarkable results in this paper. We prove, in Corollary \ref{c:EndomorphismIdeals} that, if $R$ is commutative, and $I$ is a cyclic maximal ideal which is not a direct summand, then $\End_R(I)$ and $R$ are isomorphic modulo the Jacobson radical.

In addition, we see in Proposition \ref{p:ZieglerSmallExtensions} that Ziegler small extensions are left minimal (see Definition \ref{d:ZieglerSmallExtension} for the notion of Ziegler small extensions). Since, as it was proved by Ziegler, there are many Ziegler small pure-injective extensions which are not pure-injective envelopes (see Example \ref{e:ZieglerSmallExtension}), this gives new situations in which our results can be applied.

Throughout this paper, $R$ is an associative ring with unit element. Module means right $R$-module and we denote by $\Modr R$ the category of all such modules. The Jacobson radical of $R$ is denoted by $J(R)$, and the group of units by $U(R)$. If $M$ is a module, we denote by $\End_R(M)$ its endomorphism ring and by $\Aut_R(M)$ the group  consisting of all automorphisms of $M$, that is, $\Aut_R(M) = U(\End_R(M))$.

\section{Subrings of endomorphisms rings associated to a minimal inclusion}

\noindent In this section we study the relationship between the endomorphism rings modulo the Jacobson radical of $K$ and $M$ for any left minimal inclusion $K \hookrightarrow M$. Let us begin recalling the definition of left minimal morphisms of modules \cite[p. 8]{AuslanderReitenSmalo}:

\begin{definition}\label{d:LeftMinimal}
We say that a morphism $u:M \rightarrow N$ is left minimal if any endomorphism $g:N \rightarrow N$ such that $gu=u$, is an isomorphism. If $u$ is monic, we call it a left minimal extension.
\end{definition}

\begin{examples}\label{e:ExamplesMinimalMorphisms}
\begin{enumerate}
\item Let $\mathcal X$ be a class of all modules. Recall that an $\mathcal X$-preenvelope of a module $M$ is a morphism $u:M \rightarrow C$ with $C \in \mathcal X$, such that $\Hom_R(u,C')$ is an epimorphism of abelian groups for each $C' \in \mathcal X$. An $\mathcal X$-envelope is an $\mathcal X$-preenvelope which is a left minimal morphism. The existence of $\mathcal X$-envelopes is known for many classes of modules. Examples are the injective, cotorsion and pure-injective modules.

\item Let $\mathcal I$ be an ideal in the category $\Modr R$, that is, a sub-bifunctor of the bifunctor $\Hom$. An $\mathcal I$-preenvelope of a module $M$ (see \cite{FuGuilHerzogTorrecillas}) is a morphism $u:M \rightarrow I$ belonging to $\mathcal I$, such that for each morphism $j:M \rightarrow J$ belonging to $\mathcal I$, there exists $k:I \rightarrow J$ such that $ku=j$. A monic $\mathcal I$-preenvelope is precisely an $\mathcal I$-envelope which is a left minimal morphism.

\item Recall that a submodule $K$ of a module $M$ is essential if $L \cap K \neq 0$ for each non-zero submodule $L$ of $M$. A monomorphism $u:K \rightarrow M$ is called an essential extension if $\Img u$ is an essential submodule in $M$. 

Any essential extension $u:K \rightarrow M$ with $M$ quasi-injective is left minimal. In order to prove this, take $f:M \rightarrow M$ with $fu=u$. Then $f$ is monic, since $\Ker f \cap u(K) = 0$, as $u$ is monic. Using that $M$ is quasi-injective, we get that $f$ is a split monomorphism, which implies that $\Img f$ is a direct summand of $M$ containing $u(K)$. But, as $u(K)$ is essential, $\Img f$ is the whole $M$ and $f$ is an isomorphism.

The assumption of $M$ being quasi-injective is fundamental. In order to see this, suppose that the singular right ideal $\Z(R)$ of $R$ is non-zero and not equal to $J(R)$ and let $x \in \Z(R)$ not belonging to $J(R)$. Then there exists $t \in R$ such that $1-tx$ is not a unit. Then the right ideal $I=\{r \in R \mid xr=0\}$ is essential in $R$. Moreover, the inclusion $u:I \rightarrow R$ is not left minimal, as the morphism $f:R \rightarrow R$ given by $g(z) = (1-tx)z$ satisfies that $gu=u$ and it is not an isomorphism, as $1-tx$ is not a unit.

\item As an extension of the preceding example consider the following setup: let $\mathcal F$ be an additive exact structure in $\Modr R$, that is, a class of short exact sequences in $\Modr R$ that defines an exact structure (see \cite{Buhler}), and such that it is closed under direct sums; for instance, the class of all pure-exact sequences in $\Modr R$. An $\mathcal F$-injective hull of a module $M$ is an inflation $u:M \rightarrow F$, that is, the short exact sequence $0 \rightarrow U \rightarrow F \rightarrow \Coker u \rightarrow 0$ belongs to $\mathcal F$ such that $F$ is $\mathcal F$-injective in the sense that $F$ is injective with respect to all short exact sequences belonging to $\mathcal F$ and $u$ is $\mathcal F$-essential in the sense that for any other morphism $v:F \rightarrow G$ such that $vu$ is an $\mathcal F$-inflation, $u$ is an $\mathcal F$-inflation as well. In \cite[Theorem 3.10]{CortesGuilBerkeAshish} it is proved that an $\mathcal F$-injective hull is a left minimal morphism. For instance, the pure-injective envelope of a module is left minimal.
\end{enumerate}
\end{examples}

\noindent In the preceding example we saw that an essential extension might not be left minimal. However, we have:

\begin{proposition}\label{p:FiniteLengthCokernel}
Let $u:K \rightarrow M$ be a monomorphism. Then:
\begin{enumerate}
\item If $u$ is essential and $\Coker u$ has finite length, then $u$ is left minimal.

\item If $\Coker u= \oplus_{i \in I}S_i$ for a family $\{S_i\mid i \in I\}$ of simple modules satisfying the following two conditions;
\begin{enumerate}
\item $S_i \ncong S_j$ if $i \neq j$, and 

\item $S_i$ is not isomorphic to a direct summand of $M$ for any $i\in I$.
\end{enumerate}
\noindent then $u$ is left minimal.
\end{enumerate}

\end{proposition}

\begin{proof}
(1) Let $f:M \rightarrow M$ be a morphism satisfying $fu=u$. Since $u$ is an essential monomorphism, $f$ is monic. Notice that $\Ker f \cap u(K)=0$, as $u$ is monic. Now, we can construct a commutative diagram with exact rows
\begin{displaymath}
\begin{tikzcd}
0 \arrow{r} & K \arrow{r}{u} \arrow[equal]{d} & M \arrow{r}{p} \arrow{d}{f} & \Coker u \arrow{r} \arrow{d}{\overline f} & 0\\
0 \arrow{r} & K \arrow{r}{u} & M \arrow{r}{p} & \Coker u \arrow{r} & 0
\end{tikzcd}
\end{displaymath}
Then $\overline{f}$ is monic, since $\Ker pf = f^{-1}(u(K)) = u(K) = \Ker p$. Using that $fu=u$ we get that $u(K) \leq f^{-1}(u(K))$; the other inclusion follows from the fact that if $x \in f^{-1}(u(K))$, then $f(x) = u(k)$ for some $k \in K$ and $f(x-u(k)) = 0$, which means that $x=u(k) \in u(K)$. 

Now, since $\Coker u$ has finite length, ``Fitting's Lemma" implies that $\overline f$ actually is an isomorphism. Then the ``Five Lemma" gives that $f$ is an isomorphism as well.

(2) Suppose that $u$ is the inclusion. Let $f:M \rightarrow M$ be a morphism such that $fu=u$. We can induce a commutative diagram as above
\begin{displaymath}
\begin{tikzcd}
0 \arrow{r} & K \arrow{r}{u} \arrow[equal]{d} & M \arrow{r}{p} \arrow{d}{f} & M/K \arrow{r} \arrow{d}{\overline f} & 0\\
0 \arrow{r} & K \arrow{r}{u} & M \arrow{r}{p} & M/K \arrow{r} & 0
\end{tikzcd}
\end{displaymath}
Now, consider the following subsets of $I$: $I_1=\{i \in I\mid \overline f(S_i)\neq 0\}$ and $I_2=\{i \in I \mid \overline f(S_i) = 0\}$. Then the family $\{\overline f(S_i)\mid i \in I_1\}$ is independent, since otherwise, some $S_i$ would be isomorphic to some distinct $S_j$, which contradicts the hypothesis. Using that the decomposition $\{S_i\mid i \in I\}$ complements direct summands, there exists $J \leq I$ such that $M/K=\left(\bigoplus_{i \in I_1}\overline{f}(S_i)\right) \bigoplus  \left(\bigoplus_{j \in J}S_j\right)$.  But, since $\bigoplus_{i \in I_1}\overline f(S_i) \cong \bigoplus_{i \in I_1}S_i$,  $\bigoplus_{j \in J}S_j$ is isomorphic to $\bigoplus_{i \in I_2}S_i$ which implies that $J = I_2$. The conclusion is that, if $T = \bigoplus_{i \in I_1} \overline f(S_i)$, then $M/K=T \bigoplus \left(\bigoplus_{i \in I_2}S_i\right)$.

Now take a submodule $L \leq M$ with $K \leq L$ and $L/K = \bigoplus_{i \in I_1}S_i$. Since $f$ induces the zero morphism in $M/L$, $f(M) \leq L$. Then, looking at the commutative diagram
\begin{displaymath}
\begin{tikzcd}
0 \arrow{r} & K \arrow{r} \arrow[equal]{d} & L \arrow{r}{p} \arrow{d}{f|_L} & L/K \arrow{r} \arrow{d}{\overline{f}|_{L/K}} & 0\\
0 \arrow{r} & K \arrow{r} & L \arrow{r}{p} & T \arrow{r} & 0
\end{tikzcd}
\end{displaymath}
we deduce that $f|_L$ is an isomorphism as a consequence of the ``Five Lemma". Note here that $f$ induces an isomorphism from $\oplus_{i \in I_1}S_i$ to $T$. This yields that $L$ is a direct summand of $M$, which implies that $\bigoplus_{i \in I_2}S_i$  is isomorphic to a direct summand of $M$. Since this is not the case, by hypothesis, we conclude that $I_2=0$. Consequently, $I_1 = I$, $\overline f$ is an isomorphism and, by the ``Five Lemma", $f$ is an isomorphism as well.
\end{proof}

The hypothesis of $u$ being essential in (1) cannot be removed, since the inclusion $u:K \rightarrow M$ could be, for instance, a splitting monomorphism which is not an isomorphism and splitting monomorphisms are not left minimal because, if $u:K \rightarrow M$ is such a monomorphism, taking $v:M \rightarrow K$ with $vu=1_K$, we have that $uvu=u$ but $uv:M \rightarrow M$ is not an isomorphism. Moreover, there exist non-splitting and non-essential monomorphisms with cokernel having finite length, which are not left minimal, as the following example shows.

\begin{example}
Let $M_1$ be an indecomposable module of length $2$ and $M_2$ a simple module. Let $M = M_1 \oplus M_2$. Then the inclusion $u:\textrm{Soc}(M_1) \rightarrow M$ is not essential and its cokernel has finite length. Moreover, $u$ is not left minimal, since if $e \in \End_R(M)$ is an idempotent endomorphism satisfying $e(M)=M_1$, then $eu=u$, but $e$ is not an isomorphism.
\end{example}

Other example of left minimal morphisms are given by the Ziegler small extensions introduced in \cite{Ziegler}. These small extensions are based in the notion of partial morphism introduced in \cite{Ziegler} in model theoretical language, studied in \cite{Monari} with algebraic methods and developed in \cite{CortesGuilBerkeAshish} in exact categories. In this paper we only consider partial morphism and Ziegler small extensions relative to the pure-exact structure in the module category.

\begin{definition}\label{d:ZieglerSmallExtension}
Let $u:K \rightarrow M$ be a monomorphism.
\begin{enumerate}
\item Let $f:K \rightarrow N$ be a morphism and consider the pushout of $f$ and the inclusion $u:K \rightarrow M$:
\begin{displaymath}
\begin{tikzcd}
K \arrow{r}{u} \arrow{d}{f} & M \arrow{d}{\overline f}\\
N \arrow{r}{\overline u} & P
\end{tikzcd},
\end{displaymath}
Then:
\begin{enumerate}
\item $f$ is called a \textit{partial morphism} from $M$ to $N$ with domain $K$ if $\overline u$ is a pure monomorphism. We shall denote by $\Par_R^K(M,N)$, the set of all partial morphisms from $M$ to $N$ with domain $K$.

\item $f$ is called a \textit{partial isomorphism} from $M$ to $N$ with domain $K$ if both $\overline u$ and $\overline f$ are pure monomorphisms.
\end{enumerate}
\item $u$ is a \textit{Ziegler small extension} if for any morphism $g:M \rightarrow N$ such that $gu$ is a partial isomorphism, $g$ is a pure monomorphism.
\end{enumerate}
\end{definition}

\noindent Let $u:M \rightarrow E$ be a monomorphism with $E$ pure-injective. Notice that if $u$ is a pure-injective envelope of $M$, then $u$ is a Ziegler small extension by \cite[Theorem 3.10]{CortesGuilBerkeAshish}. However, $u$ can be a Ziegler small extension without being a pure-injective envelope, that is, $u$ might not be pure. However, $u$ is always left minimal:

\begin{proposition}\label{p:ZieglerSmallExtensions}
Every Ziegler small extension $u:M \rightarrow E$ with $E$ pure-injective is a left minimal monomorphism.
\end{proposition}

\begin{proof}
Let $f:E \rightarrow E$ be such that $fu=u$. If we consider the pushout of $fu$ and $u$ we get a commutative diagram
\begin{displaymath}
\begin{tikzcd}
K \arrow{r}{u} \arrow{d}{fu} & E \arrow{d}{h_2}\\
E \arrow{r}{h_1} & P
\end{tikzcd}.
\end{displaymath}
Since $fu=u$, the identity $1_E$ satisfies that $1_Eu = 1_Efu$, so that, by the universal property of the pushout, there exists $g:P \rightarrow E$ satisfying $gh_1=1_E$ and $gh_2 = 1_E$. In particular, both $h_1$ and $h_2$ are splitting monomorphisms and, consequently, $fu$ defines a partial isomorphism from $E$ to $E$ with domain $K$. Since $u$ is a Ziegler small extension, $f$ is a pure-monomorphism.

Now, using that $E$ is pure-injective we get that $f$ is a split monomorphism and there exists $h:E \rightarrow E$ such that $hf=1_E$. Then $u=hfu=hu$ so that, using the previous argument, we conclude that $h$ is a monomorphism. Then, as $E = \Img f \oplus \Ker h$, we get that $f$ is an epimorphism.
\end{proof}

Now we define some rings of morphisms related with a monomorphism $u:K \rightarrow M$.

\begin{definition}\label{d:Subrings}
Let $u:K \rightarrow M$ be a monomorphism. We define:
\begin{itemize}
\item $\End_R^M(K) = \{f \in \End_R(K) \mid \exists g \in \End_R(M) \textrm{ with } uf=gu\}$.

\item $\End_R^K(M) = \{f \in \End_R(M) \mid \exists g \in \End_R(K) \textrm{ with } fu=ug\}$.

\item $\overline \End_R^K(M) = \{f \in \End_R(M) \mid fu=0\}$.
\end{itemize}
\end{definition}

The first of these subrings is related with partial morphisms defined previously.

\begin{proposition}\label{p:PartialMorphisms}
Let $u:K \rightarrow M$ be a monomorphism. Then:
\begin{enumerate}
\item $\End_R^K(M) \subseteq \Par_R^K(M,K)$.

\item If $M$ is pure-injective, then $\End_R^K(M)=\Par_R^K(M,K)$.
\end{enumerate}
\end{proposition}

\begin{proof}
Follows from \cite[Proposition 2.5]{CortesGuilBerkeAshish}.
\end{proof}

In order to prove the main result of this section, we need two preliminary lemmas:

\begin{lemma}\label{l:PropertiesSubrings}
Let $u:K \rightarrow M$ be a monomorphism. Then:
\begin{enumerate}
\item $\End_R^M(K)$ and $\End_R^K(M)$ are subrings of $\End_R(K)$ and $\End_R(M)$ respectively.

\item $\End_R^K(M) = \{f \in \End_R(M)\mid fu(K) \leq u(K)\}$.

\item $\overline \End_R^K(M)$ is a left ideal of $\End_R(M)$ and a two sided ideal in $\End_R^K(M)$.

\item The rings $\End_R^M(K)$ and $\End_R^K(M)/\overline \End_R^K(M)$ are isomorphic.
\end{enumerate}
\end{lemma}

\begin{proof}
(1), (2) and (3) are straightforward. In order to prove (4) consider the map $\Phi:\End_R^M(K) \rightarrow \End_R^K(M)/\overline{\End}_R^K(M)$ given by $\Phi(f) = g+\End_R^K(M)$, where $g \in \End_R(M)$ is a morphism satisfying $gu=uf$. Note that the definition of $\Phi$ does not depend on the choice of $g$, since any other morphism $f \in \End_R(M)$ verifying $hu=uf$, satisfies that $(g-h)u=0$ and, consequently, $g-h \in \overline{\End}_R^K(M)$. Moreover, $\Phi$ is epic, since any $g \in \End_R^K(M)$ satisfies that there exists $f \in \End_R^M(K)$ such that $gu=uf$ and, consequently, $\Phi(f) = g+\overline{\End}_R^K(M)$. Finally, $\Phi$ is monic because if $f \in \End^M_R(K)$ satisfies that $\Phi(f)=g+\overline\End_R^K(M)=0$, then $uf=gu=0$ and, consequently, $f=0$.
\end{proof}

Notice that, in general, if $A$ is a subring of a ring $B$, there might be no relationship between the Jacobson radicals of $A$ and $B$. In the particular case of a monomorphism $u:K \rightarrow M$, we have the following relation between $J(\End_R^K(M))$ and $J(\End_R(M))$. 

\begin{lemma}\label{l:RelationRadicals}
Let $u:K \rightarrow M$ be a monomorphism. Then:
\begin{enumerate}
\item $\End_R^K(M) \cap J(\End_R(M)) \subseteq J(\End_R^K(M))$.

\item If $u$ is left minimal, $\overline \End_R^K(M) \subseteq J(\End_R(M))$. In particular, $\overline \End_R^K(M) \subseteq J(\End_R^K(M))$.
\end{enumerate}
\end{lemma}

\begin{proof}
(1) Take $j \in \End_R^K(M) \cap J(\End_R(M))$ and let us prove that $aj$ is quasi-regular for any $a \in \End_R^K(M)$. Fix $a \in \End_R^K(M)$ and notice that, since $j \in J(\End_R(M))$, $1_M-aj$ has an inverse $t$ in $\End_R(M)$. But the equality $t(1_M-aj)=1_M$ gives that $tu=u+aju$, which implies that $tu(K) \leq u(K)$, since both $j$ and $a$ are in $\End_R^K(M)$. Consequently, $t$ actually belongs to $\End_R^K(M)$ and $j \in J(\End_R^K(M))$.

(2) It is very easy to see that $\overline \End_R^K(M)$ is a quasi-regular left ideal of $\End_R(M)$ and, in particular, it is contained in $J(\End_R(M))$: Given $f \in \End_R^K(M)$, $fu=0$ so that $(1_M-f)u=u$. Since $u$ is left minimal, $1_M-f$ is an isomorphism, that is, $f$ is quasi-regular. The last assertion follows from (1).
\end{proof}

Now we establish the main result of this section.


\begin{theorem}\label{t:IsomorphismEndomorphismRing}
Let $u:K \rightarrow M$ be a left minimal monomorphism. Then:
\begin{enumerate}
\item $\End_R^M(K)/J(\End_R^M(K))$ and $\End_R^K(M)/J(\End_R^K(M))$ are isomorphic rings.

\item $\End_R^M(K)/J(\End_R^M(K))$ is isomorphic to the subring $\pi(\End_R^K(M))$ of $\End_R(M)/J(\End_R(M))$, where $\pi$ is the canonical projection.
\end{enumerate}
\end{theorem}

\begin{proof}
(1) By Lemma \ref{l:PropertiesSubrings}, there is an isomorphism \[\Phi:\End_R^M(K) \rightarrow \End_R^K(M)/\overline{\End}_R^K(M).\]
By Lemma \ref{l:RelationRadicals}, $\overline{\End}_R^K(M)$ is contained in $J(\End_R^K(M))$, so that there exists a canonical ring epimorphism \[\Gamma:\End_R^K(M)/\overline{\End}_R^K(M) \rightarrow \End_R^K(M)/J(\End_R^K(M)).\]
Consequently, $\Gamma\Phi$ is an epimorphism from $\End_R^M(K)$ to $\End_R^K(M)/J(\End_R^K(M))$. Moreover, \[\Ker (\Gamma\Phi) = \Phi^{-1}\left(J(\End_R^K(M))/\overline{\End}_R^K(M)\right) = J(\End_R^M(K)).\]
Then, $\Gamma\Phi$ induces a ring isomorphism between $\End_R^M(K)/J(\End_R^M(K))$ and $\End_R^K(M)/J(\End_R^K(M))$.

(2) By Lemma \ref{l:PropertiesSubrings}, $\overline{\End}_R^M(K)$ is contained in $J(\End_R(M))$, so that there exists a canonical morphism \[\Gamma':\End_R^M(K)/\overline{\End}_R^M(K) \rightarrow \End_R(M)/J(\End_R(M))\]
whose image is $(\End_R^K(M)+J(\End_R(M)))/J(\End_R(M)$. Now note that $J(\End_R^K(M)/\overline{\End}_R^K(M)) \leq \Ker \Gamma'$ and that, actually, this inclusion is an equality: Given $f+\overline \End_R^K(M) \in \Ker\Gamma'$, we have that $f \in \End_R^K(M) \cap J(\End_R(M))$, which is contained in $J(\End_R^K(M))$ by Lemma \ref{l:RelationRadicals}.

Finally, $\Gamma'\Phi$ is a morphism from $\End_R^M(K)$ to $\End_R^K(M)/J(\End_R^K(M)$ whose image is $\pi(\End_R^K(M))$ and whose kernel is \[\Phi^{-1}\left(J(\End_R^K(M))/\overline{\End}_R^K(M)\right)=J(\End_R^M(K)).\]
This gives the desired isomorphism.
\end{proof}

\begin{remark} \label{r:Isomorphisms}
Note that the isomorphism \[\Psi:\End_R^M(K)/J(\End_R^M(K)) \rightarrow \End_R^K(M)/J(\End_R^K(M))\] is given by $\Psi(f+J(\End_R^M(K)) = g+J(\End_R^K(M))$, where $g$ is an endomorphism of $M$ satisfying $uf=gu$.

Analogously, the isomorphism \[\Theta:\End_R^M(K)/J(\End_R^M(K)) \rightarrow \frac{\End_R^K(M)+J(\End_R(M))}{J(\End_R(M))}\] is given by $\Theta(f+J(\End_R^M(K))) = g+J(\End_R(M))$, where $g$ is an endomorphism of $M$ satisfying $uf=gu$.
\end{remark}

Let $u:K \rightarrow M$ be a monomorphism. If $\End_R^K(M) = \End_R(M)$, then $K$ is called a fully invariant submodule of $M$. Moreover, $\End_R^M(K)=\End_R(K)$ when $u$ is $K$-injective. Recall that a morphism $f:M \rightarrow N$ in $\Modr R$ is $L$-injective for some module $L$, if $\Hom_R(f,L)$ is an epimorphism in the category of abelian groups. As an immediate consequence of Theorem \ref{t:IsomorphismEndomorphismRing} we get:

\begin{corollary}\label{c:IsomorphismFullyEnvelopes}
Let $u:K \rightarrow M$ be a monomorphism such that:
\begin{enumerate}
\item $u$ is left minimal and $K$-injective.

\item $K$ is a fully invariant submodule of $M$.
\end{enumerate}
Then $\End_R(K)/J(\End_R(K))$ and $\End_R(M)/J(\End_R(M))$ are isomorphic rings.
\end{corollary}

Examples of $K$-injective monomorphisms $u:K \rightarrow M$ are the monic preenvelopes with respect to classes of modules. Consequently:

\begin{corollary}\label{c:EndomorphismRingEnvelopes}
Let $\mathcal X$ be a class of modules and $u:M \rightarrow X$ a monic $\mathcal X$-envelope such that $M$ is a fully invariant submodule of $X$. Then $\End_R(M)/J(\End_R(M))$ and $\End_R(X)/J(\End_R(X))$ are isomorphic rings.
\end{corollary}

Modules which are fully invariant in their injective envelopes coincide with the quasi-injective modules. For modules which are fully invariant in their pure-injective envelopes or in their cotorsion envelopes we have the following result. Recall \cite[Definition 3.7]{GuilHerzog} that a monomorphism $u:K \rightarrow M$ is called \textit{strongly pure} if it is $C$-injective for every cotorsion module $C$. Clearly, strongly pure-monomorphism are pure-monomorphisms, since the existence of pure-injective envelopes implies that $u$ is pure if and only if it is $E$-injective for each pure-injective module $E$.

\begin{proposition}
\begin{enumerate}
\item Let $M$ be a module which is fully invariant in its pure-injective envelope. Then, for any pure monomorphism $u:K \rightarrow M$ and morphism $f:K \rightarrow M$, there exists $h:M \rightarrow M$ with $hu=f$.

\item Let $M$ be a module which is fully invariant in its cotorsion envelope. Then, for any strongly pure monomorphism $u:K \rightarrow M$ and morphism $f:K \rightarrow M$, there exists $h:M \rightarrow M$ with $hu=f$.
\end{enumerate}
\end{proposition}

\begin{proof}
Both are proved in a similar way. We prove (2). Let $c:M \rightarrow C$ be the cotorsion envelope of $M$, $u:K \rightarrow M$ strongly pure and $f:K \rightarrow M$. Since $u$ is strongly pure, there exists $g:M \rightarrow C$ such that $gu=cu$. Using that $c$ is a cotorsion envelope we can find $h':C \rightarrow C$ such that $h'c=g$. But, by assumption, $gc(M) \leq c(M)$, so that there exists $h':M \rightarrow M$ with $ch'=gc$. This $h'$ satisfies $h'u=f$.
\end{proof}

Theorem \ref{t:IsomorphismEndomorphismRing} allows to relate the endomorphism rings of different envelopes:

\begin{corollary}
Let $\mathcal X$ and $\mathcal Y$ be classes of modules and $M$ be a module. Let $u:M \rightarrow X$ and $v:M \rightarrow Y$ be a monic $\mathcal X$-envelope and $\mathcal Y$-envelope respectively such that $M$ is fully invariant in $X$ and $Y$. Then $\End_R(X)/J(\End_R(X))$ and $\End_R(Y)/J(\End_R(Y))$ are isomorphic rings.
\end{corollary}

\begin{proof}
Since $u$ and $v$ are preenvelopes, $\End_R^X(M) = \End_R^Y(M) = \End_R(M)$. Since $M$ is fully invariant in $X$ and $Y$, $\End_R^M(X) = \End_R(X)$ and $\End_R^M(Y) = \End_R(Y)$. Therefore, the result follows from Theorem \ref{t:IsomorphismEndomorphismRing}.
\end{proof}

Now we give another example of a $K$-injective monomorphism $u:K \rightarrow M$:

\begin{proposition}\label{p:ExtendingMorphismsCyclic}
Suppose that $R$ is commutative and let $K$ be a cyclic submodule of a free module $F$. Then the inclusion $u:K \rightarrow F$ is $K$-injective.
\end{proposition}

\begin{proof}
Suppose that $K = kR$ for some $k \in K$ and let $\{x_i \mid i \in I\}$ be a free basis of $F$. Given any $f\in \End_R(K)$, there exists $s \in R$ such that $f(k) = ks$. Consider $g \in \End_R(F)$ the unique morphism satisfying $g(x_i) = x_is$ for each $i \in I$. Then, writing $k=\sum_{i \in I}x_ir_i$, we have:
\begin{displaymath}
g(k) = \sum_{i \in I}g(x_i)r_i = \sum_{i \in I}x_is r_i = ks
\end{displaymath}
so that $gu = f$.
\end{proof}

As a consequence, we get:

\begin{corollary}\label{c:EndomorphismIdeals}
Suppose that $R$ is commutative and let $I$ be a cyclic ideal of $R$ such that $R/I=S_1 \oplus \cdots \oplus S_n$ for non-projective simple modules $S_1, \ldots, S_n$. Then $\End_R(I)/J(\End_R(I)) \cong R/J(R)$.
\end{corollary}

\begin{proof}
By Proposition \ref{p:FiniteLengthCokernel}, the inclusion $u:I \rightarrow R$ is left minimal. Notice that, as $R$ is commutative, it is verified that $S_i \ncong S_j$ for $i \neq j$. Since $R$ is commutative, $I$ is fully invariant. Finally, by Proposition \ref{p:ExtendingMorphismsCyclic}, the inclusion $u:I \rightarrow R$ is $I$-injective. Now the result follows from Corollary \ref{c:IsomorphismFullyEnvelopes}.
\end{proof}

The hypothesis of the simple modules being non-projective is crucial:

\begin{example}
Let $p$ and $q$ be distinct primes and $a = pq$. The ring $\mathbb Z/a \mathbb Z$ has two maximal ideals, $p\mathbb Z/a \mathbb Z$ and $q\mathbb Z/a \mathbb Z$, and we have $\mathbb Z/a \mathbb Z = p\mathbb Z/a \mathbb Z \oplus q \mathbb Z/a\mathbb Z$. In particular, $J\left(\mathbb Z/a \mathbb Z\right) = 0$. Moreover, $\End_R\left(p\mathbb Z/a\mathbb Z\right)$ is not isomorphic to $R$.
\end{example}

\begin{remark}
Notice that if $I$ is a maximal cyclic ideal of $R$ which is not a direct summand, then $\End_R(I)/J(\End_R(I)) \cong R/J(R)$ as a consequence of the preceding result.
\end{remark}

Since the radical of a projective module is never a direct summand, unless it is zero, we have:

\begin{corollary}
Suppose that $R$ is a commutative local ring which is not a field and such that $J(R)$ is cyclic. Then $\End_R(J(R))/J(\End_R(J(R))) \cong R/J(R)$.
\end{corollary}

\begin{remark}
Notice that the trivial situation of the preceding result is when $J(R)$ is isomorphic to $R$, for instance, if $R$ is a discrete valuation domain. However, there exist commutative local rings not satisfying these properties. For instance, $\mathbb Z_{p^n}$, for $p$ a prime number and $n$ a natural number, is a commutative local ring with non-projective Jacobson radical.
\end{remark}

When a submodule $K$ of a module $M$ is not fully invariant, we can find a fully invariant submodule of $M$ containing $K$ which is minimal with respect to these properties being fully invariant and containing $K$:

\begin{proposition}
Let $K$ be a submodule of a module $M$. We shall denote by $E_K$ the submodule of $M$ given by $\sum_{f \in \End_R(M)}f(K)$. Then $E_K$ is fully invariant and contains $K$. Moreover, $E_K$ is minimal with respect to these properties: if $L$ is a fully invariant submodule of $M$ containing $K$, then $L$ contains $E_K$.
\end{proposition}

Applying Theorem \ref{t:IsomorphismEndomorphismRing} to this situation we have:

\begin{corollary}
Let $K$ be a submodule of a module $M$ such that the inclusion $u:K \rightarrow M$ is left minimal. Then the rings $\End_R(M)/J(\End_R(M))$ and $\End_R^M(E_K)/J(\End_R^M(E_K))$ are isomorphic.
\end{corollary}

\begin{proof}
Simply note that the inclusion $E_K \hookrightarrow M$ is left minimal. Then apply Theorem \ref{t:IsomorphismEndomorphismRing}.
\end{proof}

We end this section studying submodules of pure-injective modules. As an application of the results of this section, we can describe the partial endomorphisms of a pure-injective module.

\begin{corollary}
Let $u:K \rightarrow M$ be a left minimal monomorphism with $M$ pure-injective. Then $\Par_R^K(M,K)/J(\Par_R^K(M,K))$ and $\End_R^K(M)/J(\End_R^K(M))$ are isomorphic rings.

If, in addition, $K$ is fully invariant in $M$, then $\Par_R^K(M,K)/J(\Par_R^K(M,K))$ is isomorphic to $\End_R(M)/J(\End_R(M))$.
\end{corollary}

\begin{proof}
Notice that, since $M$ is pure-injective, $\End_R^M(K) = \Par_R^K(M,K)$ by Proposition \ref{p:PartialMorphisms}. Then  apply Theorem \ref{t:IsomorphismEndomorphismRing} and Corollary \ref{c:IsomorphismFullyEnvelopes}.
\end{proof}

Notice that there does exist (many) left minimal morphisms $u:K \rightarrow M$ with $M$ pure-injective which are not pure-injective hulls, as the following example shows.

\begin{example}\label{e:ZieglerSmallExtension}
Let $K$ be a non-pure submodule of a pure-injective module $M$. Ziegler proves in \cite{Ziegler} that there exists a hull $H_M(K)$ of $K$ in $M$, that is, a pure-injective pure submodule of $M$ containing $K$ such that the inclusion $u:K \rightarrow H_M(K)$ is a maximal Ziegler small extension of $K$ in $M$ (see \cite[Theorem 1.2 and Corollary 1.3]{Monari}). Since $H_M(K)$ is a pure submodule of $M$ and is small over $K$ in $M$, $H_M(K)$ actually is small over $K$ in $H_M(K)$ by \cite[Proposition 3.3]{CortesGuilBerkeAshish}. This means that the inclusion $u:K \rightarrow H_M(K)$ is a Ziegler small extension. By Proposition \ref{p:ZieglerSmallExtensions}, $u$ is left minimal. By \cite[Proposition 1.4]{Monari}, $u$ is not a pure-injective hull of $K$ since it is not a pure-monomorphism, as $K$ is not a pure submodule of $M$.
\end{example}

\section{Automorphism invariant submodules}

\noindent In this section we study minimal monomorphisms $u:K \rightarrow M$ assuming that $K$ is invariant under automorphisms in $M$. In this case, the rings $\End_R(M)/J(\End_R(M))$ and $\End_R(K)/J(\End_R(K))$ need not be isomorphic. However, we can use Theorem \ref{t:IsomorphismEndomorphismRing} to prove that $\End_R(K)$ inherits many properties of $\End_R(M)$.

\begin{definition}
Let $u:K\rightarrow M$ be a monomorphism. We say that $K$ is automorphism invariant in $M$ if for each $f\in \Aut(M)$ there exists $g\in \End_R(K)$ such that $ug=fu$.
\end{definition}

For example of automorphism invariant submodules see \cite[Example 3.4]{GuilBerkeAshish}.

\begin{remarks}
\begin{enumerate}
\item As it is pointed out in \cite[Remark 3.2]{GuilKeskinSrivastava} the morphism $g$ in the previous definition actually is an automorphism of $K$.

\item Clearly, a right ideal $I$ of $R$ is automorphism-invariant if $U(R)I \subseteq I$.
\end{enumerate}
\end{remarks}

\noindent It is easy to find the automorphism-invariant right ideals in triangular matrix rings. Recall that a generalized triangular matrix ring is a ring of the form $T=\mat{A}{0}{X}{B}$ such that $A$ and $B$ are rings, and $X$ is a $(B,A)$-bimodule. The operations in $T$ are the usual matrix operations. The following is straightforward:

\begin{lemma}
Let $T$ be the triangular matrix ring $\mat{A}{0}{X}{B}$. Then 
\begin{displaymath}
U(T) = \left\{\mat{a}{0}{bxa}{b} \mid a \in U(A), b \in U(B), x \in X\right\}.
\end{displaymath}
\end{lemma}

\begin{proposition}
Let $T$ be the triangular matrix ring $\mat{A}{0}{X}{B}$. Then the automorphism-invariant right ideals of $T$ are of the form $\mat{I_1}{0}{I_2}{I_3}$ with
\begin{enumerate}
\item $I_1$ an automorphism-invariant right ideal of $A$.

\item $I_3$ is an automorphism-invariant right ideal of $B$.

\item $I_2$ is a right $A$-submodule of $X$ satisfying $XI_1+U(B)I_2 \subseteq I_2$.
\end{enumerate}
\end{proposition}

\begin{proof}
Let $J$ be a right ideal of $T$. If $J$ is of the form $\mat{I_1}{0}{I_2}{I_3}$, where $I_1$, $I_2$ and $I_3$ satisfy (1), (2) and (3), then $J$ is clearly automorphism-invariant. Conversely, suppose that $J$ is automorphism-invariant. It is easy to see that if $\mat{a}{0}{x}{b}\in J$, then $\mat{a}{0}{0}{0}$, $\mat{0}{0}{x}{0}$ and $\mat{0}{0}{0}{b}$ belong to $J$ as well, so that $J$ is of the form $\mat{I_1}{0}{I_2}{I_3}$ for right ideals $I_1$ and $I_3$ of $A$ and $B$ respectively, and a right $A$-submodule $I_2$ of $X$. Now, using that $J$ is automorphism-invariant it is easy to see that $I_1$, $I_2$ and $I_3$ satisfy (1), (2) and (3).
\end{proof}

The following result computes the radical of $\End_R^K(M)$ for a left minimal monomorphism $u:K \rightarrow M$ in which $K$ is automorphism-invariant.

\begin{lemma}\label{l:RadicalEndomorphismRingAutomorphismInvariant}
Let $u:K \rightarrow M$ be a left minimal monomorphism such that $K$ is automorphism-invariant in $M$. Then $J(\End_R^K(M)) = J(\End_R(M))$.
\end{lemma}

\begin{proof}
For any $j \in J(\End_R(M))$, using that $1_M-j$ is invertible and that $K$ is automorphism-invariant, we conclude that $j \in \End_R^K(M)$. Using Lemma \ref{l:RelationRadicals} we have $J(\End_R(M)) \leq J(\End_R^K(M))$. Now, the radical of $\End_R^K(M)/J(\End_R(M))$ is $J(\End_R^K(M))/J(\End_R(M))$; since $\End_R(M)/J(\End_R(M)$ has zero radical, we conclude that $J(\End_R^K(M))/J(\End_R(M))=0$. Thus it follows that $J(\End_R^K(M)) = J(\End_R(M))$.
\end{proof}

\begin{theorem}\label{t:IsomorphismEndomorphismRingAutoInvariant}
Let $u:K \rightarrow M$ be a left minimal monomorphism such that $K$ is automorphism-invariant in $M$. Then:
\begin{enumerate}
\item If idempotents in $\End_R(M)$ lift modulo $J(\End_R(M))$, then so do in $\End_R^M(K)$.

\item If $\End_R(M)/J(\End_R(M))$ is von Neumann regular and right self-injective, then $\End_R^M(K)/J(\End_R^M(K) \cong R_1 \times R_2$ where $R_1$ is an abelian regular ring and $R_2$ is a von Neumann regular right self-injective ring which is invariant under left multiplication by elements in $\End_R(M)/J(\End_R(M))$.

\item If $\End_R(M)/J(\End_R(M)$ is von Neumann and right self-injective and there do not exist nonzero ring morphisms from $\End_R(M)$ to $\mathbb Z_2$, then $K$ is fully invariant in $M$.
\end{enumerate}
\end{theorem}

\begin{proof}
(1) Let $f \in \End_R^M(K)$ be such that $f+\End_R^M(K)$ is an idempotent in $\End_R^M(K)/J(\End_R^M(K))$ and take $g$ an endomorphism of $M$ such that $gu=uf$. Then, following the notation of Remark \ref{r:Isomorphisms}, $\Theta(f+J(\End^M_R(K))) = g+J(\End_R(M))$ is idempotent in $\End_R(M)/J(\End_R(M))$. By hypothesis, there exists an idempotent $h$ in $\End_R(M)$ such that $h-g \in J(\End_R(M))$. By Lemma \ref{l:RadicalEndomorphismRingAutomorphismInvariant}, $J(\End_R(M)) = J(\End_R^K(M))$, so that there exists $t \in J(\End_R^K(M))$ such that $h-g=t$. Take $s \in \End_R(K)$ such that $us=tu$. Then $h=t+g$ and $hu=(t+g)u=u(s+f)$. Set $w=s+f$. Then $w$ is idempotent since $uw^2=h^2u=hu=uw$ and $w^2=w$ because $u$ is monic. Moreover, $u(w-f) = us$, so that, again since $u$ is monic, $w-f=s$. But $\Theta(s+J(\End_R^M(K))) = t+J(\End_R(M)) = 0$, so that $s \in J(\End_R^M(K))$. Hence $w-f \in J(\End_R^M(K))$.

(2) By Theorem \ref{t:IsomorphismEndomorphismRing}, $\End_R^M(K)/J(\End_R^M(K))$ is isomorphic to the subring $\End_R^K(M)/J(\End_R(M))$. Since $K$ is automorphism-invariant, $\End_R^K(M)/J(\End_R(M))$ is stable under left multiplication by units of $\End_R(M)/J(\End_R(M))$. Then the result follows from \cite[Theorem 2.4]{GuilKeskinSrivastava}.

(3) By \cite[Proposition 2.5]{GuilKeskinSrivastava}, $\End_R^K(M)/J(\End_R(M)) = \End_R(M)/J(\End_R(M)$, which implies that $\End_R^K(M) = \End_R(M)$. This is equivalent to $K$ being a fully invariant submodule of $M$.
\end{proof}

If the monomorphism $u:K \rightarrow M$ is $K$-injective, then $\End_R^M(K) = \End_R(K)$ and we get:

\begin{corollary} \label{c:AutomorphismInvariantMinimalKInjective}
Let $u:K \rightarrow M$ be a left minimal and $K$-injective monomorphism such that $K$ is automorphism-invariant in $M$. Then:
\begin{enumerate}
\item If idempotents in $\End_R(M)$ lift modulo $J(\End_R(M))$, then so do in $\End_R(K)$.

\item If $\End_R(M)/J(\End_R(M)$ is von Neumann regular and right self-injective, then $\End_R(K)/J(\End_R(K) \cong R_1 \times R_2$ where $R_1$ is an abelian regular ring and $R_2$ is a von Neumann regular right self-injective ring which is invariant under left multiplication by elements in $\End_R(M)/J(\End_R(M))$.
\end{enumerate}\end{corollary}

When a submodule $K$ of a module $M$ is not automorphism invariant, we can find an automorphism invariant submodule of $M$ containing $K$ which is minimal with respect to these properties being fully invariant and containing $K$.

\begin{proposition}
Let $K$ be a submodule of a module $M$. We shall denote by $A_K$ the submodule of $M$ given by $\sum_{f \in \Aut_R(M)}f(K)$. Then $A_K$ is automorphism-invariant and contains $K$. Moreover, $A_K$ is minimal with respect to this property: if $L$ is an automorphism-invariant submodule of $M$ containing $K$, then $L$ contains $A_K$.
\end{proposition}

The following well known facts about the structure of endomorphisms rings are consequences of our results:

\begin{corollary} \label{c:FinalCorollary}
Let $u:M \rightarrow E$ be a morphism.
\begin{enumerate}
\item If $u$ is an injective envelope and $M$ is automorphism-invariant in $E$, then $\End_R(M)/J(\End_R(M))$ is von Neumann regular, right self-injective and idempotents lift modulo $J(\End_R(M))$. If, in addition, $M$ is fully invariant in $E$ (equivalently, $M$ is quasi-injective), then $\End_R(M)/J(\End_R(M)) \cong \End_R(E)/J(\End_R(E))$.

\item If $u$ is a pure-injective envelope and $M$ is automorphism-invariant in $E$, then $\End_R(M)/J(\End_R(M))$ is von Neumann regular, right self-injective and idempotents lift modulo $J(\End_R(M))$. If, in addition, $M$ is fully invariant in $E$ , then $\End_R(M)/J(\End_R(M)) \cong \End_R(E)/J(\End_R(E))$.

\item If $u$ is a cotorsion envelope and $M$ is flat and automorphism-invariant in $E$, then $\frac{\End_R(M)}{J(\End_R(M))}$ is von Neumann regular, right self-injective and idempotents lift modulo $J(\End_R(M))$. If, in addition, $M$ is fully invariant in $E$, then $\End_R(M)/J(\End_R(M)) \cong \End_R(E)/J(\End_R(E))$.
\end{enumerate}
\end{corollary}

\begin{proof}
(1) The structure of $\End_R(M)$ follows from Corollary \ref{c:AutomorphismInvariantMinimalKInjective} and the structure of the endomorphism ring of an injective module (see, for instance, \cite[Theorem XIV.1.2]{Stenstrom}). The isomorphism follows from Corollary \ref{c:EndomorphismRingEnvelopes}. 

(2) The structure of $\End_R(K)$ follows from Corollary \ref{c:AutomorphismInvariantMinimalKInjective} and the structure of the endomorphism ring of a pure-injective module. The isomorphism follows from Corollary \ref{c:EndomorphismRingEnvelopes}. 

(3) Since $u$ is, in particular, a special cotorsion preenvelope, $C$ is flat and cotorsion. Then the structure of $\End_R(M)$ follows from Corollary \ref{c:AutomorphismInvariantMinimalKInjective} and the structure of the endomorphism ring of a flat cotorsion module (see \cite{GuilHerzog}). The isomorphism follows from Corollary \ref{c:EndomorphismRingEnvelopes}.
\end{proof}

\bibliographystyle{plain}
\bibliography{/home/manolo/mizurdia@ual.es/Algebra/ReferenciasBibliograficas/references}

\end{document}